\documentclass[11pt]{article}

\usepackage[margin=1.2in]{geometry}
\usepackage{tabularx}
\usepackage[T1]{fontenc}
\usepackage[utf8]{inputenc}
\usepackage[USenglish, UKenglish, english]{babel}
\usepackage{amssymb}
\usepackage{amsthm}
\usepackage{amsmath}
\usepackage[mathscr]{euscript}
\usepackage{enumitem}
\usepackage{graphicx}
\usepackage{MnSymbol}
 \let\mathscr\relax
\usepackage[scr]{rsfso}
\usepackage{tikz-cd}
\usepackage{tikz}
\usetikzlibrary{matrix,arrows,decorations.pathmorphing}
\usepackage{hyperref}
\usepackage{marginnote}
\usetikzlibrary{arrows.meta}
\usepackage{bm}
\usepackage{float}

\usepackage{cases}

\usepackage{array}

\usepackage{libertine}

\theoremstyle{definition}
\newtheorem{defin}{Definition}[section]
\theoremstyle{definition}

\theoremstyle{plain}
\newtheorem{theo}[defin]{Theorem}
\theoremstyle{plain}
\newtheorem{prop}[defin]{Proposition}
\theoremstyle{plain}
\newtheorem{lem}[defin]{Lemma}
\theoremstyle{plain}
\newtheorem{cor}[defin]{Corollary}
\theoremstyle{definition}
\newtheorem{rmk}[defin]{Remark}
\theoremstyle{definition}

\theoremstyle{definition}

\theoremstyle{plain}

\theoremstyle{definition}

\theoremstyle{definition}

\theoremstyle{plain}

\theoremstyle{plain}

\theoremstyle{plain}

\theoremstyle{definition}

\theoremstyle{plain}

\theoremstyle{definition}

\theoremstyle{definition}
\newtheorem*{defin*}{Definition}
\theoremstyle{definition}
\newtheorem*{ex*}{Example}
\theoremstyle{plain}
\newtheorem*{theo*}{Theorem}
\theoremstyle{plain}
\theoremstyle{plain}
\newtheorem*{conj*}{Conjecture}
\newtheorem*{prop*}{Proposition}
\theoremstyle{plain}
\newtheorem*{lem*}{Lemma}
\theoremstyle{plain}
\newtheorem*{cor*}{Corollary}
\theoremstyle{definition}
\newtheorem*{rmk*}{Remark}
\theoremstyle{definition}
\newtheorem*{exe*}{Exercise}

\theoremstyle{plain}
\newtheorem{theoA}{Theorem}

\theoremstyle{plain}

\theoremstyle{plain}

\theoremstyle{plain}

\theoremstyle{plain}

\theoremstyle{plain}
\newtheorem{corA}[theoA]{Corollary}

\numberwithin{equation}{section}

\usepackage{parskip}

\makeatletter
\def\thm@space@setup{%
  \thm@preskip=\parskip \thm@postskip=0pt
}
\makeatother

\setlist[enumerate]{label=(\roman*)}

\def\R{{\mathbf{R}}} 

\def\irr{{\rm Irr}}

\def\Fl{{\rm Fl}}

\def\c{{\mathbf{C}}} 
 
\def\z{{\mathbf{Z}}} 

\def\G{{\mathbf{G}}} 

\def\K{{\mathbf{K}}} 

\def\L{{\mathbf{L}}} 

\def\M{{\mathbf{M}}} 

\def\T{{\mathbf{T}}} 

\def\P{{\mathbf{P}}} 

\def\CL{{\mathcal{L}}} 

\makeatletter
\newcommand{\uset}[3][0ex]{%
  \mathrel{\mathop{#3}\limits_{
    \vbox to#1{\kern-7\ex@
    \hbox{$\scriptstyle#2$}\vss}}}}
\makeatother

\usepackage{sectsty}
\allsectionsfont{\centering}

\usepackage{xparse,etoolbox}

\newcommand{\blocktheorem}[1]{%
  \csletcs{old#1}{#1}
  \csletcs{endold#1}{end#1}
  \RenewDocumentEnvironment{#1}{o}
    {\par\addvspace{1.5ex}
     \noindent\begin{minipage}{\textwidth}
     \IfNoValueTF{##1}
       {\csuse{old#1}}
       {\csuse{old#1}[##1]}}
    {\csuse{endold#1}
     \end{minipage}
     \par\addvspace{1.5ex}}
}

\blocktheorem{theoA}
\blocktheorem{conjA}
\blocktheorem{paraA}
\blocktheorem{corA}

\theoremstyle{plain}
\newtheorem*{iAMconj*}{The inductive Alperin--McKay Conjecture}

\theoremstyle{plain}
\newtheorem*{CTC*}{The Character Triple Conjecture}

\makeatletter
\def\blfootnote{\gdef\@thefnmark{}\@footnotetext}
\makeatother

\title{
{\huge\bf Webb's conjecture and generalised Harish-Chandra theory}\\
\author{Damiano Rossi}
\date{}
\blfootnote{\emph{$2010$ Mathematical Subject Classification:} $20$G$40$, $20$C$20$, $55$P$10$, $55$P$91$.
\\
\emph{Key words and phrases:} Webb's conjecture, $p$-subgroup complex, finite reductive groups, Deligne--Lusztig theory, generalised Harish-Chandra theory.
\\
This work is supported by the Walter Benjamin Programme of the DFG - Project number 525464727. I am grateful to Gunter Malle for providing helpful comments on an earlier version of the manuscript.}
}

\begin{document}

\renewcommand{\thetheoA}{\Alph{theoA}}

\renewcommand{\thepropA}{\Alph{propA}}

\renewcommand{\theconjA}{\Alph{conjA}}

\maketitle

\begin{abstract}
Webb's conjecture states that the orbit space of the Brown complex of a finite group at any given prime $\ell$ is contractible. This conjecture was proved by Symonds in 1998. In this paper, we suggest a generalisation of Webb's conjecture for finite reductive groups. This is done by associating to each irreducible character a new simplicial complex defined in terms of Deligne--Lusztig theory. We then show that our conjecture follows from a condition, called ($e$-HC-conj) below, related to generalised Harish-Chandra theory. In particular, using earlier results of the author, we prove our conjecture and recover Symonds result for finite reductive groups under mild restrictions on the prime $\ell$. Finally, we show that the condition ($e$-HC-conj) is implied by the contractibility of the orbit spaces associated to our newly defined complex offering an unexplored topological approach to proving the uniqueness of $e$-cuspidal pairs up to conjugation.
\end{abstract}

\section{Introduction}

The Brown complex of a finite group $G$ at a prime $\ell$ is defined as the simplicial complex $\Delta(\mathcal{S}_\ell^\star(G))$ associated to the poset $\mathcal{S}_\ell^\star(G)$ of non-trivial $\ell$-subgroups of $G$ ordered by inclusion. Its study was initiated by Brown \cite{Bro75} while its importance in the study of finite group theory was highlighted by the work of Quillen \cite{Qui78}. Here, we are interested in a statement that was first conjectured by Webb. For this observe that $G$ has a natural action on the poset $\mathcal{S}_\ell^\star(G)$ induced by conjugation that endows $\Delta(\mathcal{S}_\ell^\star(G))$ with a structure of a $G$-simplicial complex. Webb conjectured in \cite{Web87} that the associated orbit space $\Delta(\mathcal{S}_\ell^\star(G))/G$ is contractible and proved in \cite{Web91} that it is always mod $\ell$ acyclic. Webb's conjecture was fully proved by Symonds \cite{Sym98} and later extended to orbit spaces of fusion systems by Linckelmann \cite{Lin09}.

Suppose now that $G=\G^F$ is a finite reductive group, for some connected reductive group $\G$ defined over an algebraically closed field of positive characteristic and some Frobenius endomorphism $F:\G\to \G$ defining an $\mathbb{F}_q$-structure on $\G$, with $q$ a prime power. Let $e$ be any positive integer and $\chi$ any irreducible complex character of $\G^F$. In this paper, we associate to each choice of $e$ and $\chi$ a $\G^F$-simplicial complex $\Delta(\CL_e^\star(\chi))$ defined in terms of Deligne--Lusztig theory (see Definition \ref{def:New complex}) in such a way that $\Delta(\CL_e^\star(1_{\G^F}))$ is homotopy equivalent to the Brown complex $\Delta(\mathcal{S}_\ell^\star(\G^F))$ whenever $\ell$ is a good prime for $\G$ not dividing $q$ and $e$ is the order of $q$ modulo $\ell$ (see Corollary \ref{cor:Comparison of orbit spaces} for a precise statement). It is then natural to ask whether the orbit space $\Delta(\CL_e^\star(\chi))/\G^F$ is contractible for every $e$ and $\chi$. Our first result reduces this problem to a well-known condition, here denoted by ($e$-HC-conj), arising in generalised Harish-Chandra theory stating that the so-called $e$-cuspidal pairs are unique up to conjugation (see Definition \ref{def:HC theory}).

\begin{theoA}
\label{thm:Main}
Let $\G^F$ be a finite reductive group and consider a positive integer $e$ and an irreducible complex character $\chi$ of $\G^F$. If the condition {\rm ($e$-HC-conj)} holds below $\chi$, then the orbit space $\Delta(\CL_e^\star(\chi))/\G^F$ is contractible.
\end{theoA}

While the condition ($e$-HC-conj) is expected to hold in general, at the present time it is only known under particular assumptions. We recall the main known cases in Remark \ref{rmk:HC theory}. In particular, by applying \cite[Proposition 4.10]{Ros24}, we obtain the following corollary. We refer the reader to Definition \ref{def:Primes, Gamma} for a description of the set of primes $\Gamma(\G,F)$ considered below.

\begin{corA}
\label{cor:Main}
Let $\G^F$ be a finite reductive group defined over $\mathbb{F}_q$ and consider a positive integer $e$. If there exists some prime $\ell\in\Gamma(\G,F)$ with $\ell\geq 5$ and such that $e$ is the order of $q$ modulo $\ell$, then the orbit space $\Delta(\CL_e^\star(\chi))/\G^F$ is contractible for every irreducible complex character $\chi$ of $\G^F$.
\end{corA}

As explained before, by considering $\chi=1_{\G^F}$ in the above corollary, we recover the contractibility of $\Delta(\mathcal{S}_\ell^\star(\G^F))/\G^F$ as conjectured by Webb and proved by Symonds. Our proof of Theorem \ref{thm:Main}, and hence of Corollary \ref{cor:Main}, relies on Bestvina--Brady discrete Morse theory \cite{Bes-Bra97}, that we apply through Bux's work \cite{Bux99}, together with techniques from Deligne--Lusztig theory and representation theory of finite reductive groups.

Before presenting our final result, we highlight the importance of the condition ($e$-HC-conj) in group representation theory. In fact, using ($e$-HC-conj) it is possible to determine a partition of the set of irreducible characters into so-called $e$-Harish-Chandra series. This partition has applications to several problems: from the parametrisation of the Brauer $\ell$-blocks in non-defining characteristic (see \cite{Bro-Mal-Mic93}, \cite{Cab-Eng94}, \cite{Cab-Eng99}, and \cite{Kes-Mal15}), to the study of the long-standing local-global conjectures (see \cite{Kes-Mal13}, \cite{Cab-Spa13}, \cite{Mal-Spa16}, \cite{Cab-Spa17I}, \cite{FLLMZ19}, \cite{Ros-Clifford_automorphisms_HC}, \cite{Ros24}, and \cite{Ros-Unip}), to even the study of the representations of $p$-adic groups (see \cite{Lan23}). Given the numerous implications of the condition ($e$-HC-conj), it would be desirable to prove it unconditionally. While this is beyond the scope of our paper, in the next result we obtain a converse of Theorem \ref{thm:Main} which gives a criterion for the validity of ($e$-HC-conj) in terms of topological properties that might be more easily verifiable. We refer the reader to Section \ref{sec:Nex complex} for the definition of $e$-pair and of the order relation $\ll_e$.

\begin{theoA}
\label{thm:Main, converse}
Let $\G^F$ be a finite reductive group and consider a positive integer $e$ and an irreducible character $\chi$ of $\G^F$. If the orbit space $\Delta(\CL_e^\star(\lambda))/\L^F$ is contractible for every $e$-pair $(\L,\lambda)\ll_e(\G,\chi)$, then the condition {\rm($e$-HC-conj)} holds below $\chi$.
\end{theoA}

The paper is organised as follows. In Section \ref{sec:Nex complex}, we introduce the relevant background material from Deligne--Lusztig theory and define the simplicial complex $\Delta(\CL_e^\star(\chi))$. Here, we also show how to recover the Brown complex $\Delta(\mathcal{S}_\ell^\star(\G^F))$ by taking $\chi=1_{\G^F}$ (see Corollary \ref{cor:Comparison of orbit spaces}). In Section \ref{sec:HC theory}, we introduce the required definitions from generalised Harish-Chandra theory and define the condition ($e$-HC-conj). Section \ref{sec:Morse theory} is devoted to discrete Morse theory. We recall the necessary results using the language of directed graphs and flag complexes as presented in \cite{Bux99}. Finally, in Section \ref{sec:Proofs} we obtain the proofs of Theorem \ref{thm:Main}, Corollary \ref{cor:Main}, and Theorem \ref{thm:Main, converse}.

\section{Simplicial complexes defined by Deligne--Lusztig theory}
\label{sec:Nex complex}

Let $\G$ be a connected reductive group defined over an algebraically closed field of positive characteristc $p$ and $F:\G\to \G$ a Frobenius endomorphism edowing the variety $\G$ with an $\mathbb{F}_q$-structure, where $q$ is a power of $p$. In their seminal papers \cite{Del-Lus76} and \cite{Lus76}, Deligne--Lusztig and Lusztig introduced a generalisation of Harish-Chandra induction defined in terms of $\ell$-adic cohomology of Deligne--Lusztig varieties. This type of induction is defined for every $F$-stable Levi subgroup $\L$ of $\G$ contained in a parabolic subgroup $\P$ and yields a functor
\[\R_{\L\leq \P}^\G:\mathbb{Z}\irr\left(\L^F\right)\to\mathbb{Z}\irr\left(\G^F\right)\]
with adjoint
\[^*\R_{\L\leq \P}^\G:\mathbb{Z}\irr\left(\G^F\right)\to\mathbb{Z}\irr\left(\L^F\right).\]
It is expected, and known in most cases, that the above functors are independent of the choice of the parabolic subgroup $\P$ (see \cite[Theorem 3.3.7 and Theorem 3.3.8]{Gec-Mal20}). All the arguments and statements considered below hold independently of the choice of the ambient parabolic subgroup and we will therefore suppress the symbol $\P$ from the above notation.

Consider now a positive integer $e$. Following \cite{Bro-Mal92}, for every $F$-stable torus $\T$ of $\G$, we denote by $\T_e$ the (unique) Sylow $e$-torus of $\T$. We then say that an $F$-stable Levi subgroup is \textit{$e$-split} if it satisfies $\L=\c_\G(\z^\circ(\L)_e)$. One can show that the $e$-split Levi subgroups of $(\G,F)$ are exactly the centralisers of $e$-tori, that is, tori satisfying $\T=\T_e$. We denote by $\CL_e(\G,F)$ the set of all $e$-split Levi subgroups and by $\CL_e^\star(\G,F)$ the subset of proper $e$-split Levi subgroups, i.e. subgroups strictly contained in $\G$. Next, we say that $(\L,\lambda)$ is an \textit{$e$-pair} of $(\G,F)$ if $\L$ is an $e$-split Levi subgroup of $(\G,F)$ and $\lambda$ is an irreducible character of $\L^F$. Following \cite[Notation 1.11]{Cab-Eng99}, we can define a binary relation, denoted by $\leq_e$, on the set of all $e$-pairs of $(\G,F)$ by setting $(\L,\lambda)\leq_e(\M,\mu)$ whenever $\L\leq \M$ and $\mu$ appears as an irreducible constituent of $\R_\L^\M(\lambda)$. To obtain a partial order relation on the set of $e$-pairs, we then consider the transitive closure $\ll_e$ of $\leq_e$. Observe that the relation $\leq_e$ is expected to be transitive, however this is not yet known in full generality (see \cite[Proposition 4.5]{Ros24} for a list of partial results). With this notation, we can now introduce the following simplicial complex.

\begin{defin}
\label{def:New complex}
Let $e$ be a positive integer and $\chi$ an irreducible character of $\G^F$. We denote by $\CL_e(\chi)$ the set of $e$-pairs $(\L,\lambda)$ of $(\G,F)$ satisfying $(\L,\lambda)\ll_e(\G,\chi)$. Furthermore, we denote by $\CL_e^\star(\chi)$ the subset of $\CL_e(\chi)$ consisting of those pairs $(\L,\lambda)$ with $\L<\G$. The set $\CL_e^\star(\chi)$ is a poset with respect to $\ll_e$ and we consider the associated order complex $\Delta(\CL_e^\star(\chi))$. Observe that a simplex $\sigma$ in $\Delta(\CL_e^\star(\G,F))$ is a totally ordered subset of $\CL_e^\star(\chi)$ with respect to $\ll_e$. In particular, there exists a minimal $e$-pair in $\sigma$ that we denote by $(\L(\sigma),\lambda(\sigma))$.
\end{defin}

In our next result, we show that the simplicial complex defined above is a natural generalisation of the complex studied in \cite{Ros-Homotopy}.

\begin{lem}
\label{lem:Comparison of posets}
The posets $\CL_e^\star(\G,F)$ and $\CL_e^\star(1_{\G^F})$ are naturally isomorphic. In particular, the complexes $\Delta(\CL_e^\star(\G,F))$ and $\Delta(\CL_e^\star(1_{\G^F}))$ are canonically homeomorphic.
\end{lem}

\begin{proof}
Consider a pair $(\L,\lambda)\in\CL_e^\star(1_{\G^F})$. By definition, we know that $(\L,\lambda)\ll_e(\G,1_{\G^F})$. Since $1_{\G^F}$ is unipotent, we deduce from \cite[Theorem 4.6.20 (b)]{Gec-Mal20} that $(\L,\lambda)\leq_e(\G,1_{\G^F})$. In other words, $1_{\G^F}$ is an irreducible constituent of the virtual character $\R_\L^\G(\lambda)$. Recalling that $\R_\L^\G$ and $^\star\R_\L^\G$ are adjoint, we obtain
\[0\neq \left\langle\R_\L^\G(\lambda),1_{\G^F}\right\rangle=\left\langle\lambda,{^\star\R_\L^\G(1_{\G^F})}\right\rangle.\]
On the other hand ${^\star\R_\L^\G(1_{\G^F})}=1_{\L^F}$ (see, for instance, \cite[Example 3.3.4]{Gec-Mal20}), and therefore we get $\lambda=1_{\L^F}$. It now follows that the map of posets $\CL_e^\star(\G,F)\to \CL_e^\star(1_{\G^F})$ given by $\L\mapsto(\L,1_{\L^F})$ is an isomorphism as required. The induced map of simplicial complexes is then a homeomorphism.
\end{proof}

Recall that each irreducible character $\chi\in\irr(\G^F)$ is invariant under $\G^F$. In particular, the conjugation action of $\G^F$ induces an action of $\G^F$ on the poset $\CL_e^\star(\chi)$. This endows $\Delta(\CL_e^\star(\chi))$ with a structure of a $\G^F$-simplicial complex (see \cite[Definition 6.1.1]{Ben98}) and we can therefore consider the corresponding orbit space $\Delta(\CL_e^\star(\chi))/\G^F$. Combining Lemma \ref{lem:Comparison of posets} with \cite[Theorem A]{Ros-Homotopy}, we can show that this orbit space generalises, as far as finite reductive groups are concerned, the orbit space associated to the Brown complex as considered in Webb's conjecture. First, we need to recall the definition of the set of primes $\pi(\G,F)$ from \cite[Definition 2.1]{Ros-Homotopy}.

\begin{defin}
\label{def:Primes, Pi}
Denote by $\pi'(\G,F)$ the set of odd primes $\ell$ that are good for $\G$, do not divide $q$ nor $|\z(\G_{\rm sc})^F|$, and with $\ell\neq 3$ when $(\G,F)$ has a rational component of type ${^3 {\bf D}}_4$. If $(\G^*,F^*)$ is in duality with $(\G,F)$ (see \cite[Definition 1.5.17]{Gec-Mal20}), then we define $\pi(\G,F)$ as the subset of those $\ell\in\pi'(\G,F)$ not dividing $|\z(\G)^F:\z^\circ(\G)^F|$ nor $|\z(\G^*)^{F^*}:\z^\circ(\G^*)^{F^*}|$.
\end{defin}

We can now show how to recover the Brown complex, and its orbit space, from the complex introduced in Definition \ref{def:New complex}.

\begin{cor}
\label{cor:Comparison of orbit spaces}
Let $\ell\in\pi(\G,F)$ not dividing the order of $\z(\G)^F$ and denote by $e=e_\ell(q)$ the order of $q$ modulo $\ell$. Then, there is a $\G^F$-homotopy equivalence
\[\Delta\left(\CL_e^\star\left(1_{\G^F}\right)\right)\simeq\Delta\left(\mathcal{S}_\ell^\star\left(\G^F\right)\right).\]
In particular, we have a homotopy equivalence of orbit spaces
\[\Delta\left(\CL_e^\star\left(1_{\G^F}\right)\right)/\G^F\simeq\Delta\left(\mathcal{S}_\ell^\star\left(\G^F\right)\right)/\G^F.\]
\end{cor}

\begin{proof}
The canonical homeomorphism from Lemma \ref{lem:Comparison of posets} yields a $\G^F$-homotopy equivalence between $\Delta(\CL_e^\star(1_{\G^F}))$ and $\Delta(\CL_e^\star(\G,F))$ and hence an equivalence of orbit spaces $\Delta(\CL_e^\star(1_{\G^F}))/\G^F\simeq \Delta(\CL_e^\star(\G,F))/\G^F$. Furthermore, by \cite[Theorem A]{Ros-Homotopy} we have a $\G^F$-homotopy equivalence between $\Delta(\CL_e^\star(\G,F))$ and $\Delta(\mathcal{S}_\ell^\star(\G^F))$ which induces an equivalence between the orbit spaces $\Delta(\CL_e^\star(\G,F))/\G^F$ and $\Delta(\mathcal{S}_\ell^\star(\G^F))/\G^F$. 
\end{proof}

\begin{rmk}
To conclude this section, we point out that the simplicial complex $\Delta(\CL_e^\star(\chi))$ can be seen as a refinement of the complex $\Delta(\CL_e^\star(B))$ introduced in \cite[Definition 3.3]{Ros-Homotopy_Brauer}, where $B$ is a Brauer $\ell$-block and $e$ the order of $q$ modulo $\ell$. In fact, under suitable assumptions, it is possible to show that $\CL_e^\star(B)$ is the union of those $\CL_e^\star(\chi)$ with $\chi$ belonging to $B$ (see \cite[Corollary 4.9]{Ros24}).
\end{rmk}

\section{Generalised Harish-Chandra theory}
\label{sec:HC theory}

Generalised Harish-Chandra theory provides a powerful tool for the study of modular representation theory of finite reductive groups in non-defining characteristic (see \cite[Chapter 3]{Gec-Mal20} and \cite[Part V]{Cab-Eng04} for an introduction to the theory). It can be roughly divided into two parts: the first part gives a partition of the irreducible characters into so-called generalised Harish-Chandra series, the second part describes each such series in terms of a relative Weyl group and its associated cyclotomic Hecke algebra (see \cite{Bro-Mal93}). In this paper, we are interested in the first of these two statements and in its connection to the contractibility of the orbit space associated with the complex introduced in Definition \ref{def:New complex}.

Let $\G$, $F$, $q$ be as in the previous section and consider a positive integer $e$. An irreducible character $\chi\in\irr(\G^F)$ is called \textit{$e$-cuspidal} if $^*\R_\L^\G(\chi)=0$ for all proper $e$-split Levi subgroups $\L<\G$. Then, we say that an $e$-pair $(\L,\lambda)$ of $(\G,F)$ is \textit{$e$-cuspidal} if $\lambda$ is an $e$-cuspidal character of $\L^F$. Observe that an $e$-pair $(\L,\lambda)$ is $e$-cuspidal if and only if it is a minimum with respect to the order relation $\ll_e$. In particular, notice that the poset $\CL_e^\star(\chi)$ from Definition \ref{def:New complex} is empty if and only if $\chi$ is $e$-cuspidal. The following condition describes the uniqueness of $e$-cuspidal pairs, up to conjugation, and is a fundamental part of generalised Harish-Chandra theory.

\begin{defin}[$e$-HC-conj]
\label{def:HC theory}
Consider an irreducible character $\chi\in\irr(\G^F)$. We say that $\chi$ satisfies ($e$-HC-conj) in $\G^F$ if all $e$-cuspidal pairs $(\L,\lambda)$ with $(\L,\lambda)\leq_e(\G,\chi)$ are $\G^F$-conjugate. Then, we say that ($e$-HC-conj) holds below $\chi$, if for every $e$-pair $(\K,\kappa)\ll_e(\G,\chi)$, the character $\kappa$ satisfies ($e$-HC-conj) in $\K^F$.
\end{defin}

While the above condition is expected for every $e\geq 1$ and any $\chi\in\irr(\G^F)$, at present it is only known under particular assumptions. Below we collect the main cases where ($e$-HC-conj) has been verified. For this we recall the definition of a set of primes closely related to $\pi(\G,F)$.

\begin{defin}
\label{def:Primes, Gamma}
Let $\gamma(\G,F)$ be the set of odd primes, good for $\G$, and not dividing $q$ nor $|\z(\G)^F:\z^\circ(\G)^F|$. Consider a pair $(\G^*,F^*)$ in duality with $(\G,F)$ and define $\Gamma(\G,F)$ to be $(\gamma(\G,F)\cap \gamma(\G^*,F^*))\setminus\{3\}$, if $\G_{\rm ad}$ has a rational component of type ${^3{\bf D}_4}$, or $\gamma(\G,F)\cap \gamma(\G^*,F^*)$ otherwise.
\end{defin}

\begin{rmk}
\label{rmk:HC theory}
Consider a positive integer $e$, a prime $\ell$ not dividing $q$, and denote by $e_\ell(q)$ the order of $q$ modulo $\ell$ (modulo $4$, when $\ell=2$). Then ($e$-HC-conj) holds below $\chi$ whenever one of the following conditions is satisfied:
\begin{enumerate}
\item $e=1$ (see \cite[Corollary 3.1.17]{Gec-Mal20});
\item $\chi$ is unipotent (see \cite[Theorem 3.2]{Bro-Mal-Mic93};
\item $\chi$ belongs to a Lusztig series associated to a semisimple element of the dual group $\G^{*F^*}$ of order prime to $\ell$, $\ell\geq 5$ does not divide $q$ (with $\ell\neq 5$ if $\G$ has components of type ${\bf E}_8$), and $e=e_\ell(q)$ (see \cite[Theorem 4.1]{Cab-Eng04});
\item $\chi$ belongs to a Lusztig series associated to a quasi-isolated semisimple element of the dual group $\G^{*F^*}$ of order prime to $\ell$ bad for $\G$, with $\G$ is simple, simply connected, and of exceptional type, and $e=e_\ell(q)$ (see \cite[Theorem 1.4 (a)]{Kes-Mal13});
\item $\ell\in\Gamma(\G,F)$ with $\ell\geq 5$, and $e=e_\ell(q)$ (see \cite[Proposition 4.10]{Ros24}). Notice here that, while the statement of Proposition \cite[Proposition 4.10]{Ros24} assumes the validity of the Mackey formula, its proof also works without this assumption simply by considering the eventual dependency on ambient parabolic subgroups.
\end{enumerate}
\end{rmk}

\section{A Morse theoretic criterion for the contractibility of orbit spaces}
\label{sec:Morse theory}

Discrete Morse theory, developed in the 1990s by Bestvina--Brady \cite{Bes-Bra97} and Forman \cite{For98}, is an adaptation of the classical Morse theory for differentiable manifolds to the setting of cell complexes. Morse theoretic techniques have been successfully exploited to obtain results on the contractibility of orbit spaces for subgroup complexes. In \cite{Bux99}, Bux used discrete Morse theory, as developed by Bestvina and Brady \cite{Bes-Bra97}, to obtain a new proof of Webb's conjecture on the contractibility of the orbit space of the $p$-subgroups complex. Recently, another proof of Webb's conjecture was given by Steinberg \cite{Ste24} utilising, instead, the discrete Morse theory of Forman \cite{For98}. In this paper, we make use of discrete Morse theory by applying a criterion obtain by Bux in \cite{Bux99}. To do so, we first recall the relevant definitions.

Let $\Gamma$ be a directed graph. For any two vertices $u,v\in \Gamma$, we write $u\leftarrow v$ if there is an edge in $\Gamma$ from $v$ to $u$. A \textit{clique} in $\Gamma$ is a subset $S$ of vertices of $\Gamma$ such that for any pair of distinct elements $u,v\in S$ we have either $u\leftarrow v$ or $v\leftarrow u$. We can then associate a simplicial complex $\Fl(\Gamma)$, called the \textit{flag complex} (or the clique complex), to $\Gamma$ as follows: the vertices of $\Fl(\Gamma)$ are those of $\Gamma$ and a simplex $\sigma$ in $\Fl(\Gamma)$ is a finite clique of $\Gamma$. Consider now a group $G$ and assume that $\Gamma$ is a directed $G$-graph, that is, $G$ acts by automorphisms of directed graphs on $\Gamma$. If $\Gamma$ is \textit{well-founded}, i.e. $\Gamma$ has no directed paths of infinite length, then $\Fl(\Gamma)$ is a $G$-simplicial complex and we can form the orbit space $\Fl(\Gamma)/G$. While the latter might not be a simplicial complex, it turns out to be a CW-complex, to which discrete Morse theory applies. Bux's \cite[Corollary 13]{Bux99}, provides a criterion for the contractibility of $\Fl(\Gamma)/G$ in terms of the action of $G$ on certain subgraphs of $\Gamma$. More precisely, for any $u\in \Gamma$, let $\Gamma_\downarrow^u$ be the subgraph of $\Gamma$ spanned by those vertices $v\in \Gamma$ such that $v\leftarrow u$. For a simplex $\sigma\in\Fl(\Gamma)$, we set $\Gamma_\downarrow^\sigma:=\bigcap_u\Gamma_\downarrow^u$, where $u$ runs over the vertices of $\sigma$. Furthermore, we say that a vertex $u$ is \textit{minimal} if $\Gamma_\downarrow^u$ is empty, or equivalently, if there are no edges in $\Gamma$ with $u$ as starting vertex. Then, we say that a simplex $\sigma\in\Fl(\Gamma)$ is \textit{minimal} if it contains a minimal vertex. With these definitions, we can finally state Bux's criterion (see \cite[Corollary 13]{Bux99}).

\begin{prop}[Bux]
\label{prop:Bux}
Let $G$ be a group and $\Gamma$ a well-founded directed $G$-graph. Suppose that $G$ acts transitively on the set of minimal vertices of $\Gamma$ and that, for each non-minimal $\sigma\in\Fl(\Gamma)$, the stabiliser $G_\sigma$ acts transitively on the set of minimal elements of $\Gamma_\downarrow^\sigma$. Then the orbit space $\Fl(\Gamma)/G$ is contractible.
\end{prop}

The above result relies on the following two statements that will also serve as the main ingredients of the proof Theorem \ref{thm:Main, converse}.

\begin{lem}[Bux]
\label{lem:Bux}
Let $\Gamma$ be a well-founded directed $G$-graph, consider the flag complex $\Fl(\Gamma)$ and suppose that $\Fl(\Gamma^\sigma_\downarrow)/G_\sigma$ is contractible for all simplices $\sigma\in\Fl(\Gamma)$. Denote by $(\Fl(\Gamma)/G)^{\leq 0}$ the discrete space whose points correspond to the $G$-orbits of minimal elements in $\Fl(\Gamma)$. Then the inclusion $(\Fl(\Gamma)/G)^{\leq 0}\hookrightarrow\Fl(\Gamma)/G$ is a homotopy equivalence.
\end{lem}

\begin{proof}
This follows immediately by \cite[Lemma 9]{Bux99} applied to the CW-complex $X=\Fl(\Gamma)/G$ after noticing that $\Fl(\Gamma^\sigma_\downarrow)/G_\sigma$ coincides with the descending link in $X$ of the image $\overline{\sigma}$ of $\sigma$ in $X$ (see \cite[Observation 10]{Bux99} and the equality in the first paragraph of \cite[p.48]{Bux99}).
\end{proof}

\begin{theo}[Bux]
\label{thm:Bux}
Let $\Gamma$ be a well-founded directed $G$-graph. Then the following are equivalent:
\begin{enumerate}
\item for every non-minimal simplex $\sigma$ in $\Fl(\Gamma)$, the orbit space $\Fl(\Gamma^\sigma_\downarrow)/G_\sigma$ is contractible.
\item for every non-minimal simplex $\sigma$ in $\Fl(\Gamma)$, the stabiliser $G_\sigma$ acts transitively on the set of minimal vertices in $\Gamma^\sigma_\downarrow$;
\end{enumerate}
\end{theo}

\begin{proof}
This follows from \cite[Theorem 12]{Bux99} after noticing the equality in the first paragraph of \cite[p.48]{Bux99}.
\end{proof}

\section{Proof of Theorem \ref{thm:Main} and Theorem \ref{thm:Main, converse}}
\label{sec:Proofs}

To prove Theorem \ref{thm:Main}, we will make use of Bux's criterion from the previous section. For this, we first need to provide a connection between the flag complex of a direct graph and the order complex of a poset. This is done by using the comparability graph. Let $\mathcal{P}$ be a poset with respect to the order relation $\leq$. The \textit{comparability graph} of $\mathcal{P}$ is the directed graph $\Gamma(\mathcal{P})$ with vertex set $\mathcal{P}$ and with an edge $u\leftarrow v$ whenever $u\leq v$ as elements in the poset $\mathcal{P}$. In the following lemma, we observe that the flag complex of the comparability graph of a poset is homeomorphic to the order complex of the poset.

\begin{lem}
\label{lem:Comparability graph}
For any poset $\mathcal{P}$ there is a canonical homeomorphism between $\Delta(\mathcal{P})$ and $\Fl(\Gamma(\mathcal{P}))$.
\end{lem} 

\begin{proof}
Let $\sigma$ be a simplex in $\Delta(\mathcal{P})$. By definition $\sigma$ is a totally ordered subset of $\mathcal{P}$. In particular, when we see $\sigma$ as a set of vertices in $\Gamma(\mathcal{P})$, we have that either $u\leftarrow v$ or $v\leftarrow u$, for any $u\neq v$ in $\sigma$. In other words, $\sigma$ is a clique in $\Gamma$ and so a simplex in $\Fl(\Gamma(\mathcal{P}))$. Conversely, if we consider a simplex $\sigma$ in $\Fl(\Gamma(\mathcal{P}))$, then we can view $\sigma$ as a totally ordered subset of $\mathcal{P}$, that is, as a simplex of $\Delta(\mathcal{P})$. Therefore, the assignment $\sigma\mapsto\sigma$ induces a homeomorphism between $\Delta(\mathcal{P})$ and $\Fl(\Gamma(\mathcal{P}))$ as desired.
\end{proof}

Thanks to Lemma \ref{lem:Comparability graph}, we can now apply Proposition \ref{prop:Bux} to the orbit space $\Delta(\CL_e^\star(\chi))/\G^F$ and prove Theorem \ref{thm:Main}.

\begin{proof}[Proof of Theorem \ref{thm:Main}]
Consider a positive integer $e$ and an irreducible character $\chi\in\irr(\G^F)$. Assuming that ($e$-HC-conj) holds below $\chi$, we want to show that the orbit space $\Delta(\CL_e^\star(\chi))/\G^F$ is contractible. Thanks to Lemma \ref{lem:Comparability graph} we know that $\Delta(\CL_e^\star(\chi))$ is canonically homeomorphic to the flag complex $\Fl(\Gamma(\CL_e^\star(\chi)))$. Therefore, it suffices to verify the hypothesis of Proposition \ref{prop:Bux} with respect to the graph $\Gamma(\CL_e^\star(\chi))$.

First, observe that $\Gamma(\CL_e^\star(\chi))$ is a directed $\G^F$-graph and that it is well-founded because its edges are induced by an order relation and $\CL_e^\star(\chi)$ is a finite set. Next, consider a minimal vertex $(\L,\lambda)$ of $\Gamma(\CL_e^\star(\chi))$. The minimality of $(\L,\lambda)$ implies that no $e$-pair $(\K,\kappa)$ can be found such that $(\K,\kappa)\ll_e(\L,\lambda)$ and $\K\neq \L$. Therefore $(\L,\lambda)$ is an $e$-cuspidal pair of $(\G,F)$. Since, by assumption, ($e$-HC-conj) holds for $\chi$, we deduce that $\G^F$ acts transitively on the minimal vertices of $\Gamma(\CL_e^\star(\chi))$.

Next, consider a non-minimal simplex $\sigma$ in $\Fl(\Gamma(\CL_e^\star(\chi)))$. We claim that there is a vertex in $\sigma$, that we denote by $(\L(\sigma),\lambda(\sigma))$, such that $(\L(\sigma),\lambda(\sigma))\leftarrow (\K,\kappa)$ for every vertex $(\K,\kappa)$ in $\sigma$. First, notice that $\sigma$ has vertices that are minimal amongst those of $\sigma$, for otherwise we could find an infinite path in $\Gamma(\CL_e^\star(\chi))$. Suppose that $(\K,\kappa)$ and $(\M,\mu)$ are two such vertices in $\sigma$. Since $\sigma$ is a clique, we know that either $(\K,\kappa)\leftarrow(\M,\mu)$ or $(\M,\mu)\leftarrow(\K,\kappa)$. This is only possible if $(\K,\kappa)=(\M,\mu)$ and so there is a unique such element as claimed. Now, it follows from \cite[Remark 11]{Bux99} that
\[\Gamma(\CL_e^\star(\chi))^\sigma_\downarrow=\Gamma(\CL_e^\star(\chi))^{(\L(\sigma),\lambda(\sigma))}_\downarrow.\]
Hence, it remains to show that the stabiliser $\G^F_\sigma$ acts transitively on $\Gamma(\CL_e^\star(\chi))^{(\L(\sigma),\lambda(\sigma))}_\downarrow$. By definition, the latter is the subgraph spanned by the vertices $(\K,\kappa)$ of $\Gamma(\CL_e^\star(\chi))$ such that $(\K,\kappa)\leftarrow(\L(\sigma),\lambda(\sigma))$ and therefore coincides with the comparability graph $\Gamma(\CL_e^\star(\lambda(\sigma)))$ of the subposet $\CL_e^\star(\lambda(\sigma))$ of $\CL_e^\star(\chi)$. Observe now that $(\L(\sigma),\lambda(\sigma))$ is not $e$-cuspidal because $\sigma$ is a non-minimal simplex in $\Fl(\Gamma(\CL_e^\star(\chi)))$. In particular, $\CL_e^\star(\lambda(\sigma))$ is non-empty. Now, as argued before, any minimal vertex $(\L,\lambda)$ of $\Gamma(\CL_e^\star(\lambda(\sigma)))$ is an $e$-cuspidal pair in $(\L(\sigma),F)$. Since ($e$-HC-conj) holds below $\chi$, it holds for $\lambda(\sigma)$ in $\L(\sigma)^F$ and therefore $\L(\sigma)^F$ acts transitively on the set of minimal vertices of $\Gamma(\CL_e^\star(\lambda(\sigma)))$. To conclude, notice that $\L(\sigma)^F$ is contained in the stabiliser $\G^F_\sigma$. In fact, for every $(\K,\kappa)$ in $\sigma$ we have $(\L(\sigma),\lambda(\sigma))\leftarrow(\K,\kappa)$ which implies that $\L(\sigma)^F\leq \K^F$. This shows that $\L(\sigma)^F$ normalises $\K^F$ and $\kappa$, for every $(\K,\kappa)$ in $\sigma$. Therefore $\G^F_\sigma$ acts transitively on the minimal vertices of $\Gamma(\CL_e^\star(\chi))^\sigma_\downarrow$ as required. The proof is now complete.
\end{proof}

As a consequence of Theorem \ref{thm:Main}, we obtain that the orbit space $\Delta(\CL_e^\star(\chi))/\G^F$ is contractible, whenever any of the conditions of Remark \ref{rmk:HC theory} is satisfied. In particular, we obtain Corollary \ref{cor:Main} as a consequence of Remark \ref{rmk:HC theory} (v). Finally, we prove Theorem \ref{thm:Main, converse}.

\begin{proof}[Proof of Theorem \ref{thm:Main, converse}]
Suppose that the statement fails to hold. Then, we can find some $e$-pair $(\K,\kappa)$ of $(\G,F)$ with $(\K,\kappa)\ll_e(\G,\chi)$ such that ($e$-HC-conj) does not hold for $\kappa$ in $\K^F$. Notice that $(\K,\kappa)$ cannot be an $e$-cuspidal pair and hence $\CL_e^\star(\kappa)$ is non-empty. We may also assume that the $e$-pair $(\K,\kappa)$ has been minimised with respect to $\dim(\K)$, so that ($e$-HC-conj) holds for $\lambda$ in $\L^F$ whenever $(\L,\lambda)\ll_e(\G,\chi)$ and $\dim(\L)<\dim(\K)$.

We now consider the flag complex $\Fl(\Gamma(\CL_e^\star(\kappa)))$ and fix one of its simplices, say $\tau$. We claim that $\Fl(\Gamma(\CL_e^\star(\kappa))^\tau_\downarrow)/\K^F_\tau$ is contractible. Without loss of generality, we may assume that $\tau$ is non-minimal. Then, the claim follows from Theorem \ref{thm:Bux}, provided that, for every non-minimal simplex $\rho$ in $\Fl(\Gamma(\CL_e^\star(\kappa)))$, the stabiliser $\K^F_\rho$ acts transitively on the set of minimal vertices of $\Gamma(\CL_e^\star(\kappa))^\rho_\downarrow$. To prove this fact, notice that, given any such simplex $\rho$, we can find a unique vertex $(\L(\rho),\lambda(\rho))$ in $\rho$ such that $(\L(\rho),\lambda(\rho))\leftarrow(\M,\mu)$ for every $(\M,\mu)$ in $\rho$ (see the argument used in the third paragraph of the proof of Theorem \ref{thm:Main}). By \cite[Remark 11]{Bux99} we get $\Gamma(\CL_e^\star(\kappa))^\rho_\downarrow=\Gamma(\CL_e^\star(\kappa))^{(\L(\rho),\lambda(\rho))}_\downarrow$ and hence the minimal vertices of $\Gamma(\CL_e^\star(\kappa))^\rho_\downarrow$ are precisely the $e$-cuspidal pairs $(\L,\lambda)$ of $(\K,F)$ satisfying $(\L,\lambda)\ll_e(\L(\rho),\lambda(\rho))$. Since $\L(\rho)< \K$, we obtain $\dim(\L(\rho))<\dim(\K)$ (see, for instance, \cite[Proposition 1.22]{Mal-Tes}) and therefore ($e$-HC-conj) holds for $\lambda(\rho)$ in $\L(\rho)^F$ by the minimality of $\K$. This in turns shows that $\K_\rho^F$ (which contains $\L(\rho)^F$) acts transitively on the set of minimal vertices of $\Gamma(\CL_e^\star(\kappa))^\rho_\downarrow$. We conclude that $\Fl(\Gamma(\CL_e^\star(\kappa))^\tau_\downarrow)/\K^F_\tau$ is contractible, as claimed above.

Now, we can apply Lemma \ref{lem:Bux} to $\Gamma(\CL_e^\star(\kappa))$ to show that the orbit space $\Fl(\Gamma(\CL_e^\star(\kappa)))/\K^F$ is homotopy equivalent to $(\Fl(\Gamma(\CL_e^\star(\kappa)))/\K^F)^{\leq 0}$. Recall that the latter is the discrete space whose points correspond to the $\K^F$-orbits of minimal vertices in $\Gamma(\CL_e^\star(\kappa))$, that is the $\K^F$-orbits of $e$-cuspidal pairs $(\L,\lambda)\ll_e(\K,\kappa)$. Since $\Fl(\Gamma(\CL_e^\star(\kappa)))/\K^F\simeq \Delta(\CL_e^\star(\kappa))/\K^F$ is contractible by hypothesis, and using Lemma \ref{lem:Comparability graph}, we conclude that $(\Fl(\Gamma(\CL_e^\star(\kappa)))/\K^F)^{\leq 0}$ consists of a single point, whence ($e$-HC-conj) holds for $\kappa$ in $\K^F$. This contradicts the choice of $(\K,\kappa)$ and the proof is now complete.
\end{proof}


\newcommand{\etalchar}[1]{$^{#1}$}

\vspace{0.5cm}

\mbox{{\sc{FB Mathematik, RPTU Kaiserslautern--Landau, Postfach 3049, 67653 Kaiserslautern, Germany}}}

\textit{Email address:} \href{mailto:damiano.rossi.math@gmail.com}{damiano.rossi.math@gmail.com}

\end{document}